\documentclass[a4paper,12pt]{article}

\usepackage{mathtools, bbm, bm,array,amsmath,amsthm}
\usepackage{tikz}
\usepackage{caption}
\usepackage{environ}
\usepackage{graphicx}
\usetikzlibrary{shapes}
\usetikzlibrary{decorations.markings}
\usetikzlibrary{patterns}

\makeatletter
\newcommand\footnoteref[1]{\protected@xdef\@thefnmark{\ref{#1}}\@footnotemark}
\makeatother

\newtheorem{theorem}{Theorem}
\newtheorem{lemma}[theorem]{Lemma}%
\newtheorem{corollary}[theorem]{Corollary}%
\newtheorem{definition}[theorem]{Definition}%

\let\originalleft\left
\let\originalright\right
\renewcommand{\left}{\mathopen{}\mathclose\bgroup\originalleft}
\renewcommand{\right}{\aftergroup\egroup\originalright}

\DeclareMathAlphabet{\pazocal}{OMS}{zplm}{m}{n}

\usepackage{pdfsync}

\usepackage{amsmath}
\usepackage{amsfonts}
\usepackage{amssymb}
\usepackage{amsthm}
\usepackage{mathrsfs}
\newtheorem*{lemma*}{Lemma}
\newtheorem*{proposition*}{Proposition}
\newtheorem*{theorem*}{Theorem}
\usepackage[normalem]{ulem}
\usepackage{extarrows}
\usepackage{xcolor}
\usepackage[margin=1cm]{caption}
\usepackage{subcaption}
\usepackage[inline]{enumitem}

\usepackage[numbers, sort]{natbib}
\DeclareUnicodeCharacter{025B}{\ensuremath{\varepsilon}}

\makeatletter
\renewcommand*\env@matrix[1][\arraystretch]{%
  \edef\arraystretch{#1}%
  \hskip -\arraycolsep
  \let\@ifnextchar\new@ifnextchar
  \array{*\c@MaxMatrixCols c}}
\makeatother

\usepackage{framed}

\newcommand{\TODO}[1][0]{%
  \ifx#10
    $\square$
  \else
    $\boxtimes$
  \fi
}

\newcommand\blfootnote[1]{%
  \begingroup
  \renewcommand\thefootnote{}\footnote{#1}%
  \addtocounter{footnote}{-1}%
  \endgroup
}

\usepackage{float}

\usepackage{tikz}
\tikzstyle{inarrow} = [<-,very thick]
\tikzstyle{bcircle} = [circle,draw = blue]
\usetikzlibrary{decorations.pathmorphing}

\usepackage[linkcolor=blue,colorlinks=true,citecolor=blue,filecolor=black]{hyperref}

\hypersetup{hyperfootnotes=false}

\usepackage{geometry}
\usepackage{setspace}

\title{Entropy of Compact Operators with Applications to Landau-Pollak-Slepian Theory and Sobolev Spaces}
\date{}
\author{Thomas Allard \\ tallard@ethz.ch  \and Helmut Bölcskei \\ hboelcskei@ethz.ch}

\begin{document}

\maketitle


\abstract{
\noindent
We derive a precise general relation between the entropy of a compact operator and its eigenvalues.
It is then shown how this result along with the underlying philosophy can be applied to improve substantially on the
best known characterizations of the entropy of the Landau-Pollak-Slepian operator and the metric entropy of unit balls in Sobolev spaces.
}



\section{Introduction}
\label{sec:sample1}

\blfootnote{The authors gratefully acknowledge support by the Lagrange Mathematics and Computing Research Center, Paris, France.}

\noindent 
Characterizing the metric entropy of function classes is a topic of longstanding interest in the mathematics and engineering literature, spanning across domains as diverse as approximation theory \cite{lorentzApproximationFunctions1966,lorentzMetricEntropyApproximation1966}, information theory \cite{yangInformationtheoreticDeterminationMinimax1999,donohoDataCompressionHarmonic1998},  statistics \cite{dudleyUniversalDonskerClasses1987,wainwrightHighDimensionalStatistics2019}, the study of dynamical systems \cite{zamesMetricComplexityCausal1977,zamesNoteMetricDimension1993}, and deep neural network theory \cite{elbrachterDeepNeuralNetwork2021,grohsPhaseTransitionsRate2021}.
Perhaps somewhat less widely known are the related concepts of entropy and entropy numbers of linear compact operators between Banach spaces \cite{carlEntropyCompactnessApproximation1990,carlInequalitiesEigenvaluesEntropy1980,edmundsFunctionSpacesEntropy1996,konigEigenvalueDistributionCompact1986}, finding application in domains as varied as control theory \cite{prosserEEntropyECapacityCertain1966}, machine learning \cite{williamson2001generalization}, and the study of Brownian motion \cite{booklifshits}.

Based on recent advances in the characterization of the metric entropy of ellipsoids by the authors of the present paper \cite{secondpaper,firstpaper}, we derive a precise relationship between the entropy of a compact operator and the asymptotic behavior of its eigenvalues, thereby improving significantly upon the classical result \cite[Proposition~1.3.2]{carlEntropyCompactnessApproximation1990}. 
As a byproduct, we also obtain a
relation between the entropy of a compact operator and its eigenvalue-counting function.



Finally, it is demonstrated how our results along with the underlying general philosophy can be applied to improve substantially on the best known characterizations of the entropy of the Landau-Pollak-Slepian operator and the metric entropy of unit balls in Sobolev spaces.

\subsection{Notation and terminology}

\noindent
We write $\mathbb{N}$ for the set of non-negative integers, $\mathbb{N}^*$ for the positive integers, $\mathbb{R}$ for the real numbers, and $\mathbb{R}^*_+$ for the positive real numbers.
For $d\in\mathbb{N}^*$, we denote by $\omega_d$ the volume of the unit ball in $\mathbb{R}^d$ and by $\pazocal{H}^d$ the $d$-dimensional Hausdorff measure.

When comparing the asymptotic behavior of the functions $f$ and $g$ as $x \to \ell$, with $\ell \in \mathbb{R}\cup\{-\infty, \infty\}$, we use the standard notation $f = o_{x \to \ell}(g)$ to express that $\lim_{x \to \ell} \frac{f(x)}{g(x)} =0$.
We further indicate the asymptotic equivalence $\lim_{x \to \ell} \frac{f(x)}{g(x)} =1$ according to
$f \stackrel{x\to\ell}{\sim} g$. 

 The space of square summable sequences is referred to as $\ell^2$. 
 Given a set $\Omega \subseteq \mathbb{R}^d$, we let $L^2(\Omega)$ be the space of square-integrable functions on $\Omega$ equipped with the usual inner product $\langle \cdot  ,  \cdot \rangle_{L^2(\Omega)}$.
We write $\text{supp}(f)$ for the essential support of  $f \in L^2(\Omega)$
and define $C_0^\infty(\Omega)$ to be the space of infinitely differentiable functions with compact support contained in $\Omega$.
Further, $\text{Id}$ designates the identity operator and $\pazocal{F}$ the Fourier transform operator. 
Given a Banach space $(E, \|\cdot\|_E)$, we shall write 
$\pazocal{B}_E \coloneqq \left\{x \in E \mid \|x\|_E\leq 1\right\}$
for its closed unit ball.

Finally, $\log(\cdot)$ stands for the logarithm to base $2$, $\ln(\cdot)$ is the natural logarithm, and $\mathbbm{1}_{X}(\cdot)$ designates the indicator function corresponding to the set $X$.

\section{Entropy of compact operators}\label{sec:entrcompop}

\noindent
We first introduce the notions of metric entropy and entropy numbers of a set.

\begin{definition}[Metric entropy and entropy numbers of sets]\label{def:entropyandnbset}
Let $(\pazocal{X}, d)$ be a metric space and $\pazocal{K}\subseteq \pazocal{X}$ a compact set. 
An $\varepsilon$-covering of $\pazocal{K}$ with respect to the metric $d$ is a set $\{x_1,...\, ,x_N\} \subseteq \pazocal{X}$  such that for each $x \in \pazocal{K}$, there exists an $i\in \{1, \dots, N\}$ so that $d(x,x_i)\leq \varepsilon$. The $\varepsilon$-covering number $N(\varepsilon ; \pazocal{K}, d)$ is the cardinality of a smallest such $\varepsilon$-covering.
The \emph{metric entropy} of the set $\pazocal{K}$ is given by
\begin{equation*}
H \left(\varepsilon ; \pazocal{K}, d\right)
\coloneqq \log N\left(\varepsilon ; \pazocal{K}, d\right).
\end{equation*}
For $m\in \mathbb{N}^*$, the \emph{$m$-th entropy number} $\varepsilon_m$ of $\pazocal{K}$ is defined as the smallest radius $\varepsilon>0$ required to cover $\pazocal{K}$ with at most $2^m$ balls of radius $\varepsilon$, i.e.,
\begin{equation*}
\varepsilon_m \left(\pazocal{K}, d\right) 
\coloneqq \inf \left\{\varepsilon >0 \mid H \left(\varepsilon ; \pazocal{K}, d\right) \leq m\right\}.
\end{equation*}
\end{definition}

\noindent
The extension of the concepts in Definition \ref{def:entropyandnbset} to compact linear operators $T\colon E \to F$ between Banach spaces $E$, $F$ 
is as follows. 
We first note that the image $T(\pazocal{B}_E)$ of the unit ball in $E$ has compact closure in $F$. 
The entropy of the compact operator $T$, quantifying its compactness, is then simply given by the metric entropy of $\overline{T(\pazocal{B}_E)}$.

\begin{definition}[Entropy and entropy numbers of compact linear operators]
Let $(E, \|\cdot\|_E)$ and $(F, \|\cdot\|_F)$ be Banach spaces, and let $T\colon E \to F$ be a compact linear operator.
We define the entropy of $T$ as the metric entropy of the closure of the image of the unit ball $\pazocal{B}_E$, according to
\begin{equation*}
H \left(\varepsilon ; T \right)
\coloneqq H \left(\varepsilon ; \overline{T(\pazocal{B}_E)}, \|\cdot\|_F \right).
\end{equation*}
Likewise, for $m\in \mathbb{N}^*$, the $m$-th entropy number of 
$\, T$ is defined as the $m$-th entropy number of the closure of the image of the unit ball $\pazocal{B}_E$, i.e.,
\begin{equation*}
\varepsilon_m(T) \coloneqq
\inf \left\{\varepsilon >0 \mid H \left(\varepsilon ; \overline{T(\pazocal{B}_E)}, \|\cdot\|_F \right) \leq m\right\}.
\end{equation*}
\end{definition}

\noindent
Following the convention in the literature, we shall talk about the ``metric entropy'' of a set but will simply say the ``entropy'' of an operator.
Although all the results in this paper can be stated in the broader setting of operators between general Banach spaces, for concreteness and simplicity of exposition, we restrict our attention to endomorphisms of separable real Hilbert spaces. 
The restriction we impose further has the advantage of allowing direct comparisons with classical results, as, e.g., those in \cite{carlEntropyCompactnessApproximation1990, prosserEEntropyECapacityCertain1966}.
We shall henceforth consider a separable real Hilbert space $\mathcal{H}$ along with the compact linear operator $T\colon \mathcal{H}\to \mathcal{H}$.
Arguing through polar decomposition as in \cite{prosserEEntropyECapacityCertain1966} or \cite[Chapter~3.4]{carlEntropyCompactnessApproximation1990}, we can further restrict our attention to positive self-adjoint compact operators. Such operators can be diagonalized (in suitable bases); the corresponding eigenvalues are positive and will be denoted as $\{\lambda_n\}_{n\in \mathbb{N}^*}$.
For expositional convenience, eigenvalues will be assumed ordered in a non-increasing manner throughout the paper.

Relating the entropy numbers of compact linear operators to their eigenvalues has been a topic of longstanding interest.
The corresponding results in the area are typically lower and upper bounds on the entropy numbers in terms of the geometric mean of the eigenvalues, see \cite{carlInequalitiesEigenvaluesEntropy1980} for Banach spaces, \cite[Theorem~1.3.4]{edmundsFunctionSpacesEntropy1996} for quasi-Banach spaces, and \cite{prosserEEntropyECapacityCertain1966} for Hilbert spaces.
To the best of our knowledge, the sharpest known result \cite[Proposition~1.3.2]{carlEntropyCompactnessApproximation1990} is
\begin{equation}\label{eq:bestlitbound}
\sup_{N \in \mathbb{N}^*} \left\{2^{-m/N} \left[\prod_{n=1}^N \lambda_n \right]^{1/N}\right\}
\leq \varepsilon_m(T)
\leq 6 \sup_{N \in \mathbb{N}^*} \left\{2^{-m/N} \left[\prod_{n=1}^N \lambda_n \right]^{1/N}\right\},
\end{equation}
for all $m\in \mathbb{N}^*$.

Our approach is based on the observation that the image of the unit ball in $\mathcal{H}$ under $T$ is an ellipsoid with semi-axes given by the eigenvalues $\{\lambda_n\}_{n\in \mathbb{N}^*}$ of $T$.
The strategy of using the metric entropy of ellipsoids to characterize the entropy numbers of compact linear operators has been employed previously 
in the literature, see, e.g., \cite{prosserEEntropyECapacityCertain1966}. 
However, recent progress on the characterization of the metric entropy of ellipsoids \cite{secondpaper,firstpaper}, by the authors of the present paper, allows for significant improvements. In particular, we obtain the following general result.

\begin{theorem}\label{cor: Metric entropy for polynomial ellipsoids second order}
Let $\mathcal{H}$ be a separable real Hilbert space and let $T\colon \mathcal{H}\to \mathcal{H}$ be a positive self-adjoint compact operator with eigenvalues $\{\lambda_n\}_{n\in\mathbb{N}^*}$ satisfying
\begin{equation}\label{eq:asympeigen}
\lambda_n 
= \frac{c_1}{n^{\alpha_1}} + \frac{c_2}{n^{\alpha_2}} + o_{n \to \infty} \left(\frac{1}{n^{\alpha_2}}\right),
\end{equation}
where $c_1\in\mathbb{R}_+^*$, $c_2\in\mathbb{R}$, and $\alpha_1, \alpha_2 \in\mathbb{R}_+^* $ are such that
\begin{equation*}
\text{either } \quad 
\alpha_1<\alpha_2 < \alpha_1+1/2,
\ \text{ or } \quad
\begin{cases}
\alpha_1 = \alpha_2, \text{ and } \\
c_2 = 0.
\end{cases} 
\end{equation*}
Then, the entropy of $T$ satisfies
\begin{equation}\label{eq:main1}
H \left(\varepsilon;T\right) 
=  \frac{\alpha_1{c_1}^{\frac{1}{\alpha_1}}}{\ln(2)} \, {\varepsilon}^{-\frac{1}{\alpha_1}}
+  \frac{c_2\, {c_1}^{\frac{1-\alpha_2}{\alpha_1}}}{\ln(2)(\alpha_1-\alpha_2+1)} \, \varepsilon^{-\frac{\alpha_1-\alpha_2+1}{\alpha_1}} 
+  o_{\varepsilon \to 0}\left(\varepsilon^{-\frac{\alpha_1-\alpha_2+1}{\alpha_1}}\right),
\end{equation}
which can equivalently be expressed in terms of entropy numbers according to 
\begin{equation}\label{eq:main2}
\varepsilon_m (T) 
=  c_1  \left(\frac{\alpha_1}{\ln(2)}\right)^{\alpha_1}m^{-\alpha_1} 
+ \frac{c_2 }{\alpha_1-\alpha_2+1} \left(\frac{\alpha_1}{\ln(2)}\right)^{\alpha_2} {m^{-\alpha_2}} 
+ o_{m \to \infty} \left({m^{-\alpha_2}}\right).
\end{equation}
\end{theorem}

\begin{proof}
As $\mathcal{H}$ is a separable real Hilbert space and $T$ a (positive) self-adjoint compact operator, there exists an orthonormal basis $\{\psi_n\}_{n\in\mathbb{N}^*}$ of $\mathcal{H}$ composed of eigenvectors of $T$  (see, e.g., \cite[Theorem~6.11]{brezisFunctionalAnalysisSobolev2011}).
Using the Bessel-Parseval identity, one obtains the following characterization of the image, under $T$, of the unit ball $\pazocal{B}$ in $\mathcal{H}$:
\begin{align*}
T(\pazocal{B})
&= \left\{y\in \mathcal{H}   \mid y=Tx \text{ with } \|x\|^2_\mathcal{H} \leq 1 \right\}\\
&= \left\{y\in \mathcal{H}   \mid y= \sum_{n=1}^\infty \lambda_n x_n \psi_n \text{ with } \{x_n\}_{n\in\mathbb{N}^*} \in \ell^2 \text{ s.t. } \sum_{n=1}^\infty |x_n|^2 \leq 1 \right\}\\
&= \left\{y\in \mathcal{H}   \mid y= \sum_{n=1}^\infty y_n \psi_n \text{ with } \{y_n\}_{n\in\mathbb{N}^*} \in \ell^2 \text{ s.t. }  \sum_{n=1}^\infty |y_n/\lambda_n|^2 \leq 1 \right\}.
\end{align*}
This shows that $T(\pazocal{B})$ is isometric to the $\ell^2$-ellipsoid with semi-axes $\{\lambda_n\}_{n\in\mathbb{N}^*}$.
The result (\ref{eq:main1}) then follows as a direct consequence of Lemma~\ref{lem:corpolydecay} in the Appendix.

Now, turning to (\ref{eq:main2}), we first note that, by the definition of $\varepsilon_m(T)$, 
\begin{equation}\label{eq:oqsnknvdbsvfdnsfres}
H \left(\varepsilon_m(T);T\right) 
\leq m 
< H \left(\varepsilon_m(T)-\eta;T\right),
\quad \text{for all } m\in\mathbb{N}^* \text{ and } \eta>0.
\end{equation}
Using (\ref{eq:main1}) and choosing $\eta$ small enough, we deduce from (\ref{eq:oqsnknvdbsvfdnsfres}) that
\begin{equation*}
m
=  \frac{\alpha_1{c_1}^{\frac{1}{\alpha_1}} }{\ln(2)} \, {\varepsilon_m}^{-\frac{1}{\alpha_1}} 
+  \frac{c_2\, {c_1}^{\frac{1-\alpha_2}{\alpha_1}}}{\ln(2)(\alpha_1-\alpha_2+1)} \, \varepsilon_m^{-\frac{\alpha_1-\alpha_2+1}{\alpha_1}} 
+  o_{m \to \infty}\left(\varepsilon_m^{-\frac{\alpha_1-\alpha_2+1}{\alpha_1}}\right).
\end{equation*}%
Inverting this expression, by application of Lemma~\ref{lem:inv1} in the case $\alpha_1=\alpha_2, c_2=0$ and Lemma~\ref{lem:inv} otherwise, yields (\ref{eq:main2}).
\end{proof}

\noindent
In view of the applications considered in Section \ref{sec:appl}, we decided to restrict the statement of Theorem~\ref{cor: Metric entropy for polynomial ellipsoids second order} to regularly varying (in the sense of \cite[Definition 1.2.1]{binghamRegularVariation1987}) eigenvalue sequences 
$\{\lambda_n\}_{n\in\mathbb{N}^*}$. 
An extension to exponentially decaying eigenvalue sequences can be obtained by
replacing the argument in the proof of Theorem~\ref{cor: Metric entropy for polynomial ellipsoids second order} relying on Lemma~\ref{lem:corpolydecay} by \cite[Theorem~9]{firstpaper}.

As announced, we now show how Theorem~\ref{cor: Metric entropy for polynomial ellipsoids second order} leads to a significant improvement of (\ref{eq:bestlitbound}). 
To this end, we consider the case $\lambda_n \stackrel{n\to\infty}{\sim }c_1 n^{-\alpha_1}$
and note that application of Stirling's formula allows to characterize the integer $N$ attaining the supremum in (\ref{eq:bestlitbound}), resulting in
\begin{equation}\label{eq:nrjkzhgekjrgbhtdukiytfgf}
\sup_{N \in \mathbb{N}^*} \left\{2^{-m/N} \left[\prod_{n=1}^N \lambda_n \right]^{1/N} \right\}
= c_1  \left(\frac{\alpha_1}{\ln(2)}\right)^{\alpha_1}m^{-\alpha_1} + o_{m\to\infty}\left({m^{-\alpha_1}}\right).
\end{equation}
A detailed derivation of (\ref{eq:nrjkzhgekjrgbhtdukiytfgf}) is provided in Lemma~\ref{lem:jkdqvbftgzdezrfgezg} in the Appendix.
We can conclude from (\ref{eq:nrjkzhgekjrgbhtdukiytfgf}) that (\ref{eq:bestlitbound}) characterizes the first-order term only in the asymptotic expansion of $\varepsilon_m (T)$ and does so up to a multiplicative factor of $6$.
Because of this gap between the upper and the lower bound, (\ref{eq:bestlitbound}) cannot provide a characterization of $\varepsilon_m (T)$ beyond the first-order term. In contrast, our result (\ref{eq:main2}) specifies the first- and second-order terms with precise constants.

We note that Theorem~\ref{cor: Metric entropy for polynomial ellipsoids second order} does not presuppose a full characterization of the eigenvalues of $T$, but only requires information on their asymptotic behavior.
This is compatible with 
standard results in micro-local analysis (see, e.g., \cite[Chapter~30]{shubinPseudodifferentialOperatorsSpectral2001}) characterizing the asymptotic  behavior of the eigenvalue-counting function
\begin{equation*}
M_T \colon \gamma \in (0,\infty) \mapsto \#\left\{ n \in \mathbb{N}^* \mid \lambda_n \geq \gamma \right\} \in \mathbb{N}
\end{equation*}
of (pseudo-)differential operators $T$ in the asymptotic regime $\gamma \to 0$.
The canonical example is given by the Laplacian (relevant, e.g., in the study of Sobolev spaces), where the Weyl law leads to the asymptotic scaling of the eigenvalue-counting function (see Section~\ref{sec:applsob} for details).
The following result is a simple consequence of Theorem~\ref{cor: Metric entropy for polynomial ellipsoids second order} and illustrates how knowledge of the asymptotic behavior of $M_T(\gamma)$ yields an asymptotic characterization of the entropy and entropy numbers of $T$.

\begin{corollary}\label{cor: Metric entropy for polynomial ellipsoids second order ecf}
Let $\mathcal{H}$ be a separable real Hilbert space and let $T\colon \mathcal{H}\to \mathcal{H}$ be a positive self-adjoint compact operator with eigenvalue-counting function satisfying
\begin{equation}\label{eq:asympevf}
M_T(\gamma) 
= \kappa_1\gamma^{-\beta_1} + \kappa_2 \gamma^{-\beta_2} + o_{\gamma \to 0} \left(\gamma^{-\beta_2}\right),
\end{equation}
where $\kappa_1\in\mathbb{R}_+^*$, $\kappa_2\in\mathbb{R}$,  and $\beta_1, \beta_2 \in\mathbb{R}_+^*$ are such that 
\begin{equation}\label{eq:condonthebetas}
\text{either } \quad
\beta_1/2<\beta_2 < \beta_1, \ \text{ or } \quad 
\begin{cases}
\beta_1=\beta_2, \text{ and } \\
\kappa_2 = 0.
\end{cases}
\end{equation}
Then, the entropy of $T$ satisfies
\begin{equation*}
H \left(\varepsilon; T\right)
= \frac{\kappa_1}{\beta_1\ln(2)}
{\varepsilon}^{-\beta_1} 
+ \frac{\kappa_2}{\beta_2\ln(2)}
\varepsilon^{-\beta_2} 
+ o_{\varepsilon \to 0}\left(\varepsilon^{-\beta_2}\right),
\end{equation*}
which, upon letting $\beta^* \coloneqq \beta_1/(1+\beta_1-\beta_2)$, can equivalently be expressed as
\begin{equation*}
\varepsilon_m (T) 
=   \left(\frac{{\kappa_1}}{{\beta_1}\ln(2)}\right)^{1/\beta_1} \! m^{-1/\beta_1}
+ \frac{\kappa_2}{\kappa_1 \beta_2 }\left(\frac{{\kappa_1}}{{\beta_1}\ln(2)}\right)^{1/\beta^*}  m^{-1/\beta^*} 
+ o_{m \to \infty} \left(m^{-1/\beta^*} \right).
\end{equation*}
\end{corollary}

\begin{proof}
Denote the eigenvalues of $T$, ordered in non-increasing fashion, by $\{\lambda_n\}_{n\in\mathbb{N}^*}$ and let $\{\eta_n\}_{n\in\mathbb{N}^*}$ be an arbitrary sequence of positive real numbers.
From the definition of the eigenvalue-counting function, it follows that
\begin{equation*}
M_T(\lambda_n+\eta_n) < n \leq M_T(\lambda_n),
\quad \text{for all } n\in\mathbb{N}^*.
\end{equation*}
Invoking assumption (\ref{eq:asympevf}) on the eigenvalue-counting function and choosing $\{\eta_n\}_{n\in\mathbb{N}^*}$ to decay to zero fast enough, we then obtain
\begin{equation*}
n 
=\kappa_1\lambda_n^{-\beta_1} + \kappa_2 \lambda_n^{-\beta_2} + o_{n \to \infty} \left(\lambda_n^{-\beta_2}\right).
\end{equation*}
Inverting this expression by application of Lemma~\ref{lem:inv1} in the case where $\beta_1=\beta_2$ and $\kappa_2=0$, and Lemma~\ref{lem:inv} otherwise, yields
\begin{equation*}
\lambda_n
= \kappa_1^{1/\beta_1}\, n^{-1/\beta_1}
+ \frac{\kappa_1^{1/\beta_1-\beta_2/\beta_1}\kappa_2}{\beta_1}\, n^{\beta_2/\beta_1-1/\beta_1-1} 
+ o_{n \to \infty} \left(n^{\beta_2/\beta_1-1/\beta_1-1}\right).
\end{equation*}
Next, with a view towards application of Theorem~\ref{cor: Metric entropy for polynomial ellipsoids second order}, introducing the quantities
\begin{equation}\label{eq:jkvbfkjzfbezz}
c_1 \coloneqq \kappa_1^{1/\beta_1}, \ \
c_2 \coloneqq \frac{\kappa_1^{1/\beta_1-\beta_2/\beta_1}\kappa_2}{\beta_1}, \ \
\alpha_1 \coloneqq  1/\beta_1, \ \text{ and } \
\alpha_2 \coloneqq  1+1/\beta_1-\beta_2/\beta_1,
\end{equation}
allows to reformulate the two cases in (\ref{eq:condonthebetas}) as
\begin{equation*}
\text{either } \quad 
\alpha_1<\alpha_2 < \alpha_1+1/2,
\ \text{ or } \quad
\begin{cases}
\alpha_2 = \alpha_1, \text{ and } \\
c_2 = 0.
\end{cases} 
\end{equation*}
The hypotheses of Theorem~\ref{cor: Metric entropy for polynomial ellipsoids second order} are thus verified and its application, with $c_1$, $c_2$, $\alpha_1$, and $\alpha_2$ as defined in (\ref{eq:jkvbfkjzfbezz}), yields the desired result according to
\begin{equation*}
H \left(\varepsilon; T\right)
= \frac{\kappa_1}{\beta_1\ln(2)}
{\varepsilon}^{-\beta_1} 
+ \frac{\kappa_2}{\beta_2\ln(2)}
\varepsilon^{-\beta_2} 
+ o_{\varepsilon \to 0}\left(\varepsilon^{-\beta_2}\right)
\end{equation*}
and 
\begin{equation*}
\varepsilon_m (T) 
=  \left(\frac{{\kappa_1}}{{\beta_1}\ln(2)}\right)^{1/\beta_1}  m^{-1/\beta_1}
+ \frac{\kappa_2}{\kappa_1 \beta_2 }\left(\frac{{\kappa_1}}{{\beta_1}\ln(2)}\right)^{1/\beta^*}  m^{-1/\beta^*} 
+ o_{m \to \infty} \left(m^{-1/\beta^*} \right),
\end{equation*}
where $\beta^* \coloneqq \beta_1/(1+\beta_1-\beta_2)$.
\end{proof}

\section{Applications}\label{sec:appl}

We now put the general results developed in Section \ref{sec:entrcompop} to work. 
Concretely, we derive the entropy of the Landau-Pollak-Slepian operator 
and we find a precise asymptotic characterization of the metric entropy of unit balls in Sobolev spaces.
In both cases, significant improvements over the best known results in the literature are obtained.

\subsection{Entropy of the Landau-Pollak-Slepian operator}\label{sec:sampthm}

\noindent
The classical sampling theorem \cite{shannonCommunicationPresenceNoise1949} quantifies the minimum number of samples per unit of time needed to recover a 
strictly band-limited signal. 
This result essentially characterizes the information rate of band-limited signals.
Landau, Pollak, and Slepian \cite{landauProlateSpheroidalWave1961,slepianProlateSpheroidalWave1961} took this idea further by allowing for signals that are effectively band-and time-limited. 
The object of central interest in this theory is the Landau-Pollak-Slepian operator defined as follows.
For $r\in\mathbb{R}_+^*$ and compact subsets $\Omega$ and $\pazocal{W}$ of $\mathbb{R}^d$, one considers the sets
\begin{align*}
    &\pazocal{D}(r \Omega) 
    \coloneqq \left\{f \in L^2(\mathbb{R}^d) \mid \text{supp} (f) \subseteq r \Omega \right\} \quad \text{and } \\
    &\pazocal{F}(\pazocal{W}) 
    \coloneqq \left\{f \in L^2(\mathbb{R}^d) \mid \text{supp} \left(\pazocal{F}f\right) \subseteq  \pazocal{W} \right\}.
\end{align*}
Associating an orthogonal projection operator with each of these sets according to
\begin{equation*}
P_{r\Omega} \colon f \mapsto \mathbbm{1}_{\{r\Omega\}} f
\quad \text{and} \quad 
P_{\pazocal{W}} \colon f \mapsto \pazocal{F}^{-1} \mathbbm{1}_{\{\pazocal{W}\}} \pazocal{F} f,
\end{equation*}
leads to the definition of the Landau-Pollak-Slepian operator as
\begin{equation*}
P^{(r)}_{LPS}
\coloneqq P_{r\Omega} P_{\pazocal{W}} P_{r\Omega} \colon  L^2(\mathbb{R}^d)   \to  L^2(\mathbb{R}^d) .
\end{equation*}
We refer to \cite[Chapter~2]{daubechiesTenLecturesWavelets1992} and \cite[Chapter~20]{wongWaveletTransformsLocalization2002} for in-depth material on the Landau-Pollak-Slepian operator and to \cite{donoho1989uncertainty} for its role in the derivation of uncertainty principles. 

Next, we apply the method developed in Section~\ref{sec:entrcompop} to obtain
an exact characterization of the entropy rate
$\lim_{r\to \infty} H(\varepsilon; P^{(r)}_{LPS})/r^d$ of the Landau-Pollak-Slepian operator,
based on which an asymptotic result by Kolmogorov and Tikhomirov on the entropy rate of
effectively band- and time-limited signals can be turned into a non-asymptotic statement.
To the best of our knowledge, the entropy (rate) of the Landau-Pollak-Slepian operator has not been characterized before
in the literature.


\begin{theorem}\label{thm: LPS and sampling}
Let $d\in\mathbb{N}^*$ and let $\Omega$ and $\pazocal{W}$ be compact subsets of $\mathbb{R}^d$.
Then, we have
\begin{equation}\label{eq:khjvzebrhbkjfdnsfnz}
\lim_{r\to \infty} \frac{H \left(\varepsilon; P^{(r)}_{LPS} \right)}{r^d}
= \frac{2 \, \pazocal{H}^d(\Omega) \, \pazocal{H}^d(\pazocal{W})}{(2\pi)^d} \log \left(\varepsilon^{-1}\right),
\quad \text{for all } \varepsilon \in (0,1].
\end{equation}
\end{theorem}

The proof of Theorem~\ref{thm: LPS and sampling} 
proceeds by relating the problem at hand to that of covering an infinite-dimensional ellipsoid with semi-axes determined by the eigenvalues of the Landau-Pollak-Slepian operator. 
Indeed, it is well known (see, e.g., \cite[Lemma~1 and Theorem~1]{landauSzegoEingenvalueDistribution1975}) that the Landau-Pollak-Slepian operator is compact, self-adjoint, and has non-negative\footnote{
    Note that the Landau-Pollak-Slepian operator may have eigenvalues equal to zero. In contrast, 
    the general results in Section~\ref{sec:entrcompop} apply to positive (self-adjoint compact) operators.
    This does, however, not constitute any technical problems as here we shall only need aspects of the results in Section~\ref{sec:entrcompop} that
    do not require strict positivity of the eigenvalues.
}
eigenvalues $\{\lambda_n\}_{n\in\mathbb{N}^*}$ no larger than $1$ with the associated eigenvalue-counting function satisfying
    \begin{equation}\label{eq: eigenvalue rep lps}
    M_r(\gamma) = \left(\frac{r}{2\pi}\right)^d \pazocal{H}^d(\Omega) \, \pazocal{H}^d(\pazocal{W}) + o_{r \to \infty}\left(r^d\right), \quad \text{for all } \gamma \in (0,1),
    \end{equation}
where the dependence on $\gamma$ of the right-hand side is hidden in the $o$-term. 

\begin{center}
\begin{tikzpicture}[scale=1]
    \draw[->] (-1,0)-- (11,0);
    \draw[->] (0,-0.25) node[below] {0} -- (0,2.5) ;

    \draw[dashed, red!75!black] (-.1,.75) node[left] {$\gamma$} -- (10.5,.75);
    \draw (-.1,2) node[left] {1} -- (.1,2);

    \draw (.5,-0.25) node[below] {1} -- (.5,.25) ;  
    \draw (1,-0.25) node[below] {2} -- (1,.25) ;  
    \draw (1.5,-0.25) node[below] {3} -- (1.5,.25) ;  
    \draw (2,-0.25) node[below] {4} -- (2,.25) ;  
    \draw (2.5,-0.25) node[below] {$\cdots$} -- (2.5,.25) ;  
    \draw (6,-0.25) node[below] {$M_r(\gamma)$} -- (6,.25) ;

    \draw[fill=green!75!black, fill opacity=.5, draw opacity=.5] (.5,1.99) circle (0.075);
    \draw[fill=green!75!black, fill opacity=.5, draw opacity=.5] (1,1.985) circle (0.075);
    \draw[fill=green!75!black, fill opacity=.5, draw opacity=.5] (1.5,1.98) circle (0.075);
    \draw[fill=green!75!black, fill opacity=.5, draw opacity=.5] (2,1.975) circle (0.075);
    \draw[fill=green!75!black, fill opacity=.5, draw opacity=.5] (2.5,1.97) circle (0.075);
    \draw[fill=green!75!black, fill opacity=.5, draw opacity=.5] (3,1.965) circle (0.075);
    \draw[fill=green!75!black, fill opacity=.5, draw opacity=.5] (3.5,1.96) circle (0.075);
    \draw[fill=green!75!black, fill opacity=.5, draw opacity=.5] (4,1.955) circle (0.075);
    \draw[fill=green!75!black, fill opacity=.5, draw opacity=.5] (4.5,1.95) circle (0.075);
    \draw[fill=green!75!black, fill opacity=.5, draw opacity=.5] (5,1.9) circle (0.075);
    \draw[fill=green!75!black, fill opacity=.5, draw opacity=.5] (5.5,1.6) circle (0.075);
    \draw[fill=green!75!black, fill opacity=.5, draw opacity=.5] (5.75,1.2) circle (0.075);
    \draw[fill=green!75!black, fill opacity=.5, draw opacity=.5] (6,.8) circle (0.075);
    \draw[fill=green!75!black, fill opacity=.5, draw opacity=.5] (6.25,.5) circle (0.075);
    \draw[fill=green!75!black, fill opacity=.5, draw opacity=.5] (6.55,.25) circle (0.075);
    \draw[fill=green!75!black, fill opacity=.5, draw opacity=.5] (7,.12) circle (0.075);
    \draw[fill=green!75!black, fill opacity=.5, draw opacity=.5] (7.5,.08) circle (0.075);
    \draw[fill=green!75!black, fill opacity=.5, draw opacity=.5] (8,.06) circle (0.075);
    \draw[fill=green!75!black, fill opacity=.5, draw opacity=.5] (8.5,.05) circle (0.075);
    \draw[fill=green!75!black, fill opacity=.5, draw opacity=.5] (9,.045) circle (0.075);
    \draw[fill=green!75!black, fill opacity=.5, draw opacity=.5] (9.5,.04) circle (0.075);
    \draw[fill=green!75!black, fill opacity=.5, draw opacity=.5] (10,.035) circle (0.075);

\end{tikzpicture}

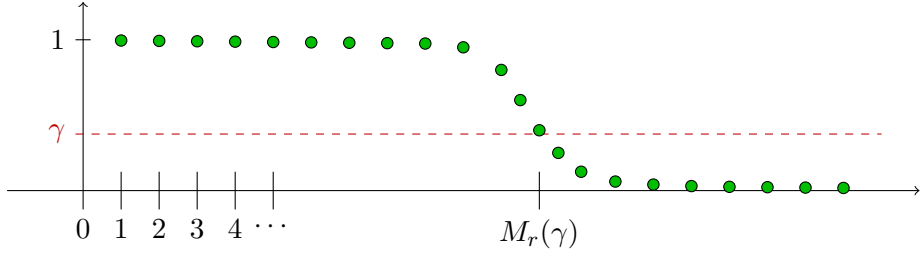
\captionof{figure}{Eigenvalue distribution of the Landau-Pollak-Slepian operator.}
\label{fig:LPSeig}
\end{center}

In a second step, the problem is then reduced to covering a finite-dimensional ellipsoid obtained by carefully 
thresholding the infinite-dimensional ellipsoid under consideration. 
Specifically, the threshold specifies the $\varepsilon$-dependent effective dimension of the infinite-dimensional ellipsoid and is chosen as $M_r(\gamma)$ for $\gamma$ suitably depending on $\varepsilon$, see Figure 1. 
The proof is completed by applying results from \cite{firstpaper} on the covering of finite-dimensional ellipsoids.

\begin{proof}
    If either $\pazocal{H}^d(\Omega)=0$ or $\pazocal{H}^d(\pazocal{W})=0$, the Landau-Pollak-Slepian operator vanishes 
    and there is nothing left to prove as both sides in (\ref{eq:khjvzebrhbkjfdnsfnz}) are equal to zero.
    In what follows, we can therefore assume that $\Omega$ and $\pazocal{W}$ both have non-zero $d$-dimensional Hausdorff measure.

    The proof will be effected by sandwiching the entropy rate 
        between matching lower and upper bounds.
    To this end, let us fix 
    $\varepsilon\in(0,1]$, $\gamma\in(0,1)$ and
    denote the image of the unit ball in $L^2(\mathbb{R}^d)$ under the operator $P^{(r)}_{LPS}$ as $\pazocal{E}^{(r)}$.
        By the same arguments as used in the proof of Theorem~\ref{cor: Metric entropy for polynomial ellipsoids second order},
    the set $\pazocal{E}^{(r)}$ is isometric to the infinite-dimensional $\ell^2$-ellipsoid with semi-axes 
    $\{\lambda_n\}_{n\in\mathbb{N}^*}$. 
    We henceforth identify $\pazocal{E}^{(r)}$ with this ellipsoid and let
    $\pazocal{E}^{(r)}_-$ stand for the finite-dimensional ellipsoid obtained from $\pazocal{E}^{(r)}$ by retaining the $M_r(\gamma)$ 
    largest semi-axes. 
    Note that $M_r(\gamma)$ is guaranteed to be non-zero for $r$ large enough and recall that we are interested in the large-$r$ limit.
    As covering the infinite-dimensional ellipsoid $\pazocal{E}^{(r)}$ requires at least as many covering balls as needed to cover the corresponding finite-dimensional ellipsoid $\pazocal{E}^{(r)}_-$, we have
    \begin{equation}\label{eq: tkjkjehzgsmmqdpre}
    H \left(\varepsilon; \pazocal{E}^{(r)}, \|\cdot\|_2 \right)
    \geq H \left(\varepsilon; \pazocal{E}^{(r)}_-, \|\cdot\|_2 \right).
    \end{equation}
    Now, a direct application of Lemma~\ref{lem:lbentropy} in the Appendix, with $d=M_r(\gamma)$, yields
    \begin{equation}\label{eq: tkjkjehzgsmmqd}
    \frac{H \left(\varepsilon; \pazocal{E}^{(r)}_-, \|\cdot\|_2 \right)}{2 M_r(\gamma)}
    \geq \log\left(\varepsilon^{-1}\right) + \frac{1}{M_r(\gamma)}\sum_{n=1}^{M_r(\gamma)} \log\left(\lambda_n\right) .
    \end{equation}
    Note that, by definition of $M_r(\gamma)$, the eigenvalues $\lambda_n$ appearing in (\ref{eq: tkjkjehzgsmmqd}) are all greater than or equal to $\gamma$.
    This yields the bound 
    \begin{equation}\label{eq: jjhgfhjhjjkosqqq1}
    \frac{1}{M_r(\gamma)}\sum_{n=1}^{M_r(\gamma)} \log\left(\lambda_n\right)
    \geq \log (\gamma).
    \end{equation}
    Now, combining (\ref{eq: eigenvalue rep lps})-(\ref{eq: jjhgfhjhjjkosqqq1}), we get 
    \begin{align*}
    \lim_{r\to \infty} \frac{H \left(\varepsilon; \pazocal{E}^{(r)}, \|\cdot\|_2 \right)}{r^d}
    &\geq \lim_{r\to \infty} \frac{H \left(\varepsilon; \pazocal{E}_-^{(r)}, \|\cdot\|_2 \right)}{r^d}\\
    &\geq 2 \log\left(\gamma \, \varepsilon^{-1}\right) \lim_{r\to \infty} \frac{ M_r(\gamma)}{r^d} \\
    & = \frac{2 \, \pazocal{H}^d(\Omega) \, \pazocal{H}^d(\pazocal{W}) }{(2\pi)^d}\log\left(\gamma \, \varepsilon^{-1}\right).
    \end{align*}
    In particular, taking $\gamma$ arbitrarily close to $1$, it follows that
    \begin{align}\label{eq: lowerboundonthepls}
    \lim_{r\to \infty} \frac{H \left(\varepsilon; \pazocal{E}^{(r)}, \|\cdot\|_2 \right)}{r^d}
    \geq \frac{2 \, \pazocal{H}^d(\Omega) \, \pazocal{H}^d(\pazocal{W}) }{(2\pi)^d}\log\left(\varepsilon^{-1}\right).
    \end{align}

    \noindent
    The proof will be completed by establishing an upper bound on the entropy rate matching the lower bound (\ref{eq: lowerboundonthepls}).
    To this end, fix ${\tau\in(0,1)}$ and consider the 
    ellipsoid $\pazocal{E}^{(r)}_+$ obtained from $\pazocal{E}^{(r)}$ by retaining the $M_r(\tau \varepsilon)$ largest semi-axes.
    Note that $M_r(\tau \varepsilon)$ is well-defined by $\tau\varepsilon\in(0,1)$, and, again,
    guaranteed to be non-zero for $r$ large enough.
    As the semi-axes corresponding to the dimensions not retained in the transition from $\pazocal{E}^{(r)}$ to $\pazocal{E}^{(r)}_+$ have  length smaller than $\tau \varepsilon$,
    every $(1-\tau)\varepsilon$-covering of $\pazocal{E}^{(r)}_+$ can be turned into an $\varepsilon$-covering of $\pazocal{E}^{(r)}$, simply by completing the components of the covering ball centers of $\pazocal{E}^{(r)}_+$ by an infinite sequence of zeros.
    This observation translates into 
    \begin{equation}\label{eq: tkjkjehzgsmmqdprett}
    H \left(\varepsilon; \pazocal{E}^{(r)}, \|\cdot\|_2 \right)
    \leq H \left((1-\tau)\varepsilon; \pazocal{E}^{(r)}_+, \|\cdot\|_2 \right)
    \leq H \left((1-\tau)\varepsilon; \pazocal{B}^{(r)}, \|\cdot\|_2 \right),
    \end{equation}
    where $\pazocal{B}^{(r)}$ denotes the unit ball in $\mathbb{C}^{M_r(\tau \varepsilon)}$ and the second inequality follows from the fact that the semi-axes are given by the eigenvalues of $P^{(r)}_{LPS}$ and hence have length smaller than or equal to one. 
    From Lemma~\ref{lem:densityarg} in the Appendix with $d= {M_r(\tau\varepsilon)}$, we can deduce the existence of 
    a sequence $\kappa$ with $\lim_{M \to \infty} \kappa(M) = 1$, such that
   \begin{equation*}
   N\left((1-\tau)\varepsilon; \pazocal{B}^{(r)}, \|\cdot\|_2 \right)^{1/(2 M_r(\tau \varepsilon))}  
   \leq \frac{\kappa\left({M_r(\tau\varepsilon)}\right)}{\left(1-\tau\right) \,\varepsilon  } .
   \end{equation*}
   Taking logarithms, we obtain
    \begin{equation}\label{eq: tkjkjehzgsmmqdtt}
    \frac{H \left((1-\tau)\varepsilon; \pazocal{B}^{(r)}, \|\cdot\|_2 \right)}{2 M_r(\tau \varepsilon)}
    \leq \log\left(\left[(1-\tau)\varepsilon\right]^{-1}\right) + \log\left(\kappa\left({M_r(\tau\varepsilon)}\right)\right).
    \end{equation}
    Using (\ref{eq: tkjkjehzgsmmqdtt}) in (\ref{eq: tkjkjehzgsmmqdprett}) together with (\ref{eq: eigenvalue rep lps}), it follows, for all $\tau\in(0,1)$, that
    \begin{align*}
    \lim_{r\to \infty} \frac{H \left(\varepsilon; \pazocal{E}^{(r)}, \|\cdot\|_2 \right)}{r^d}
    &\leq \lim_{r\to \infty} \frac{H \left((1-\tau)\varepsilon; \pazocal{B}^{(r)}, \|\cdot\|_2 \right)}{r^d}\\
    &\leq \lim_{r\to \infty} \frac{2  M_r(\tau \varepsilon)}{r^d} 
    \left\{\log\left(\left[(1-\tau)\varepsilon\right]^{-1}\right) + \log\left(\kappa\left({M_r(\tau\varepsilon)}\right)\right) \right\}\\
    &= \frac{2 \, \pazocal{H}^d(\Omega) \, \pazocal{H}^d(\pazocal{W})}{(2\pi)^d} \log \left(\left[(1-\tau)\varepsilon\right]^{-1}\right),
    \end{align*}
where the last step is by $\lim_{M \to \infty} \kappa(M) =1$ combined with $\lim_{r \to \infty} M_r(\gamma) = \infty$. 
We can finally choose $\tau$ arbitrarily small to obtain
        \begin{equation*}
    \lim_{r\to \infty} \frac{H \left(\varepsilon; \pazocal{E}^{(r)}, \|\cdot\|_2 \right)}{r^d} 
    \leq \frac{2 \, \pazocal{H}^d(\Omega) \, \pazocal{H}^d(\pazocal{W})}{(2\pi)^d} \log \left(\varepsilon^{-1}\right).
    \end{equation*}
    This concludes the proof.
\end{proof}

We now discuss the implications of Theorem~\ref{thm: LPS and sampling} and, for simplicity, take $d=1$,
$\Omega=[-1,1]$, and $\pazocal{W}=[-\sigma, \sigma],\sigma \in \mathbb{R}_{+}^*$. 
In this case, the image of the unit ball in $L^2(\mathbb{R})$ under $P^{(T)}_{LPS}$, with $T \in \mathbb{R}^*_+$, is obtained by 
localizing strictly band-limited (namely to $\pazocal{W}$) functions to the time-interval 
$[-T,T]$.
We denote the resulting function class by $B^{(T)}_{\sigma}$ and recall the following result due to Kolmogorov and Tikhomirov \cite[Chapter~7, Theorem~XXII]{shiryayevSelectedWorksKolmogorov1993}, \cite[Theorem~8]{jagermanEEntropyApproximationBandlimited1969}
\begin{equation}\label{eq: LPS me sampling tikh2}
    \lim_{T\to \infty} \frac{H \left(\varepsilon; B^{(T)}_{\sigma}, \|\cdot\|_2 \right)}{2T}
\stackrel{\varepsilon\to 0}{\sim} \frac{2 \sigma}{\pi} \log \left(\varepsilon^{-1}\right),
\end{equation}
which says that asymptotically, namely as $\varepsilon\to 0$, the number of degrees of freedom per unit of time of signals band-limited to $\pazocal{W}=[-\sigma, \sigma]$ 
is ${2 \sigma}/{\pi}$ and hence determined by the bandwidth $2\sigma$. 
The rationale behind this interpretation in terms of degrees of freedom derives itself from the observation that
the metric entropy of finite intervals on the reals is of order $\log(\varepsilon^{-1})$, as 
$\varepsilon \rightarrow 0$. 
Therefore, the factor $2\sigma/\pi$ on the right-hand side of (\ref{eq: LPS me sampling tikh2}) quantifies the number of information-carrying scalars per unit of time. 
We emphasize that this interpretation applies only asymptotically in $\varepsilon$. 
Consequently, for $\varepsilon \in (0,1]$, in principle the effective number of degrees of freedom could depend on $\varepsilon$, of course in a manner 
ensuring compatibility with the asymptotics in (\ref{eq: LPS me sampling tikh2}).

The result in Theorem~\ref{thm: LPS and sampling} above, i.e.,
\begin{equation}\label{eq: LPS me sampling tikh}
    \lim_{T\to \infty} \frac{H \left(\varepsilon; B^{(T)}_{\sigma}, \|\cdot\|_2 \right)}{2T}
= \frac{2 \sigma}{\pi} \log \left(\varepsilon^{-1}\right),
\quad \text{for all } \varepsilon \in (0,1],
\end{equation}
hence constitutes a substantial improvement over the literature as it proves equality in (\ref{eq: LPS me sampling tikh2}) for all $\varepsilon \in (0,1]$.
Specifically, this shows that the dimension counting argument that has been employed in the literature for decades is, in fact, exact for all $\varepsilon \in (0,1]$.



\subsection{Metric entropy of unit balls in Sobolev spaces}\label{sec:applsob}

\noindent
Standard results on the metric entropy of unit balls in function spaces defined through regularity constraints take the form
\begin{equation*}
c \, \varepsilon^{-d/k}\left(1+o_{\varepsilon\to 0} (1)\right)
\leq H \left( \varepsilon ; \mathcal{F}_k, \|\cdot\|_2 \right)
\leq C \, \varepsilon^{-d/k}\left(1+o_{\varepsilon\to 0} (1)\right),
\end{equation*}
for some constants $c, C>0$,
where $d \in \mathbb{N}^*$ stands for the dimension of the domain and $k\in \mathbb{N}^*$ is the degree of smoothness.
We now consider the case of unit balls $\mathcal{F}_k$ in Sobolev spaces 
and show how Theorem~\ref{cor: Metric entropy for polynomial ellipsoids second order} and Corollary~\ref{cor: Metric entropy for polynomial ellipsoids second order ecf} lead to a full characterization of the first-order term in the asymptotic expansion of the metric entropy of $\mathcal{F}_k$, that is, we determine the exact value of $c=C$.
Furthermore, we also obtain the second-order term under mild regularity constraints on the domain. 
Interestingly, while it is known that the first-order term is proportional to the volume of the domain (see \cite[Theorem~20]{secondpaper}), 
our result shows that the second-order term is proportional to its perimeter.

Concretely, for a bounded open set $\Omega \subset \mathbb{R}^d$, we consider the Sobolev space $W^{k,2}_0(\Omega)$ of order $k$ 
(i.e., the closure of $C^\infty_0(\Omega)$ in $L^2(\Omega)$ with respect to the topology induced by the norm (\ref{eq:defnormsob}) below; see, e.g., \cite[Chapter~3]{adamsSobolevSpaces2003} or
\cite[Chapter~9.4]{brezisFunctionalAnalysisSobolev2011}) equipped with the norm
\begin{equation}\label{eq:defnormsob}
\left\| \cdot \right\|_{k, \Omega} 
\colon f \longmapsto \left[\left\|f\right\|^2_{L^2(\Omega)}+\sum_{|\alpha|=k} \left\|D^{\alpha}f\right\|^2_{L^2(\Omega)}\right]^{1/2}, 
\end{equation}
where we used the standard multi-index notation, that is,  for $\alpha=(\alpha_1, \dots, \alpha_d)$, we have $|\alpha|=\sum_{j=1}^{d}\alpha_{j}$ and $D^\alpha = \partial_1^{\alpha_1} \dots \, \partial_d^{\alpha_d}$.
The best known result is due to Donoho (see \cite{donohoCountingBitsKolmogorov2000} or the remark after \cite[Corollary 2.4]{luschgySharpAsymptoticsKolmogorov2004}) and pertains to the one-dimensional case.
Specifically, for
 $\Omega = (0,2\pi)$, Donoho quantifies
 the constant in the leading term of the asymptotic expansion of the metric entropy $H (\varepsilon; \mathcal{F}_k, \| \cdot \|_{L^2(\Omega)}  )$
 of the unit ball $\mathcal{F}_k \coloneqq \{ f \in W^{k,2}_0(\Omega) \mid \left\| f \right\|_{k, \Omega} \leq 1 \}$ 
 according to
\begin{equation}\label{eq:donohostateoftheart}
H \left(\varepsilon; \mathcal{F}_k, \left\| \cdot \right\|_{L^2(\Omega)}  \right)
\stackrel{\varepsilon\to 0}{\sim} \frac{2 \, k}{\ln(2)} \, \varepsilon^{-\frac{1}{k}}.
\end{equation}
In the following, we simply write $H (\varepsilon)$ for $H (\varepsilon;\mathcal{F}_k, \left\| \cdot \right\|_{L^2(\Omega)}  )$. 
In Theorem~\ref{thm:sob} below, we extend \eqref{eq:donohostateoftheart} in two aspects.
First, we allow for general bounded open domains $\Omega \subset \mathbb{R}^d$ in arbitrary finite dimensions $d$. 
Second, we provide---under certain regularity constraints on $\Omega \subset \mathbb{R}^d$---an exact characterization of the second-order term in the asymptotic expansion of $H (\varepsilon)$.
Both of these extensions have no counterpart in the existing literature.

Our result relies on a spectral analysis of the Laplacian ${-\Delta \coloneqq -\partial_1^2 - \dots - \partial_d^2}$.
More concretely, we resort to Weyl's law for the Laplacian to characterize the asymptotic behavior of its eigenvalue-counting function for bounded domains $\Omega$ with smooth boundary $\partial \Omega$ and such that the measure of all periodic billiards is zero.
We refer to \cite{ivriiSecondTermSpectral1980}, \cite[Corollary~29.3.4]{hormanderAnalysisLinearPartial2009}, \cite[Chapter~1.2]{geisingerSemiclassicalLimitDirichlet2011}, and the survey \cite{Ivrii_2016} for a detailed discussion of Weyl's law and the technical condition
we require on $\partial \Omega$.
To the best of our knowledge, Weyl's law has not appeared before in the context of metric entropy of function classes.

\begin{theorem}\label{thm:sob}
Let $d, k \in \mathbb{N}^*$ and let $\Omega \subset \mathbb{R}^d$ be a bounded open subset of $\mathbb{R}^d$.
For a given set $S\subset \mathbb{R}^d$, we define a rescaled version of its $r$-dimensional Hausdorff measure $\pazocal{H}^r(S)$, $r\in\{1, \dots, d\}$,
according to
\begin{equation}\label{eq:contentofset}
\chi_r (S) \coloneqq \frac{\omega_r }{r \, (2\pi)^r \, \ln(2)}\,  \pazocal{H}^r(S),
\end{equation}
with $\omega_r$ the volume of the unit ball in $\mathbb{R}^r$.
Then, 
\begin{enumerate}[label=(\roman*)]
\item
the metric entropy and the entropy numbers of the unit ball in $W^{k,2}_0(\Omega)$ equipped with the norm (\ref{eq:defnormsob}) satisfy
\begin{equation*}
H \left(\varepsilon\right) 
= k\, \chi_d(\Omega) \, \varepsilon^{-\frac{d}{k}} + o_{\varepsilon \to 0}\left(\varepsilon^{-\frac{d}{k}}\right)
\end{equation*}
and
\begin{equation*}
\varepsilon_m
= {\left(k \, \chi_d (\Omega)\right)^{\frac{k}{d}}}\, m^{-\frac{k}{d}} + o_{m\to \infty}\left({m^{-\frac{k}{d}}}\right);
\end{equation*}

\item
if we further assume that $d \geq 3 $ and the boundary $\partial \Omega$ is smooth and such that the measure of the periodic billiards in $\Omega$ is zero, we have
\begin{equation*}
H \left(\varepsilon\right) 
= k\, \chi_d(\Omega) \, {\varepsilon}^{-\frac{d}{k}}  
- \frac{k\, \chi_{d-1}(\partial \Omega)}{4}\, 
\varepsilon^{-\frac{d-1}{k}} 
+ o_{\varepsilon \to 0}\left(\varepsilon^{-\frac{d-1}{k}}\right),
\end{equation*}
and
\begin{equation*}
\varepsilon_m 
= {\left(k \, \chi_d (\Omega)\right)^{\frac{k}{d}}}\, m^{-\frac{k}{d}}
 \ - \frac{k\,  \chi_{d-1}(\partial \Omega)}{4\, d\, \chi_d (\Omega)}\,\,{\left(k \, \chi_d (\Omega)\right)^{\frac{k+1}{d}}} m^{-\frac{k+1}{d}}
+ o_{m\to \infty}\left({m^{-\frac{k+1}{d}}}\right).
\end{equation*}
\end{enumerate}
\end{theorem}

\noindent
The statements in Theorem~\ref{thm:sob} hold identically when
the norm (\ref{eq:defnormsob}) is replaced by the equivalent norm
\begin{equation}\label{eq:defnormsob2}
\left\| \cdot \right\|_{k, \Omega}'
\colon f\longmapsto \left[\sum_{|\alpha|=k} \left\|D^{\alpha}f\right\|^2_{L^2(\Omega)} \right]^{1/2}.
\end{equation}
We refer to \cite[Corollary 6.31]{adamsSobolevSpaces2003} for a proof of the equivalence of the norms (\ref{eq:defnormsob}) and (\ref{eq:defnormsob2}).
Our choice to state Theorem~\ref{thm:sob} in terms of the norm (\ref{eq:defnormsob}) is motivated by the desire to be compatible with Donoho's result (\ref{eq:donohostateoftheart}).
Indeed, for $d=1$ and $\Omega=(0,2\pi)$, 
part (i) of Theorem~\ref{thm:sob} recovers (\ref{eq:donohostateoftheart}) according to
\begin{equation*}
    H(\varepsilon) 
     \stackrel{\varepsilon \to 0}{\sim} k\, \chi_1 ((0,2\pi)) \, \varepsilon^{-\frac{1}{k}} 
    \stackrel{\varepsilon \to 0}{\sim} k\, \frac{2\cdot 2\pi}{1 \cdot (2\pi)^1 \ln(2)} \, \varepsilon^{-\frac{1}{k}} 
\stackrel{\varepsilon\to 0}{\sim} \frac{2\, k}{\ln(2)} \, \varepsilon^{-\frac{1}{k}} .
\end{equation*}


\begin{proof}
We first introduce the positive self-adjoint compact operator
\begin{equation}\label{eq:defT}
T \coloneqq \left[\text{Id} + (-\Delta)^{(k)}\right]^{-{1/2}},
\end{equation}
where $(-\Delta)^{(k)}$ stands for the $k$-fold application of the Laplacian, 
and note that, for all $f \in C_0^{\infty}(\Omega)$,
\begin{equation}\label{eq:norm-equivalence}
\|f\|_{k, \Omega} = \sqrt{\left\langle T^{-1} f , T^{-1} f \right\rangle_{L^2(\Omega)}}.
\end{equation}
By density of $C_0^{\infty}(\Omega)$ in $W^{k,2}_0(\Omega)$ (see \cite[Chapter~9, Remark~18]{brezisFunctionalAnalysisSobolev2011}), we conclude from
(\ref{eq:norm-equivalence}) that the unit ball under the Sobolev norm $\|\cdot\|_{k, \Omega}$ is an ellipsoid in $L^2(\Omega)$ with semi-axes given by the 
eigenvalues of $T$.
Next, let 
\begin{equation*}
\widetilde{M}_{\Delta} \colon \gamma  \in (0,\infty)
\mapsto \# \left\{n \in \mathbb{N}^* \mid \lambda_n \leq \gamma, \text{ where } \{ \lambda_n\}_{n \in \mathbb{N}^*} \text{ are the eigenvalues of } -\Delta  \right\}.
\end{equation*}
Note that $\widetilde{M}_{\Delta}$ counts the number of eigenvalues of $-\Delta$ below a given threshold, whereas $M_T$ counts those above. 
From the definition of $T$ in (\ref{eq:defT}), it follows that
\begin{equation*}
M_T (\gamma)
= \widetilde{M}_{\Delta} \left(\sqrt[k]{\gamma^{-2}-1}\right), \quad \text{for all } \gamma \in (0,1).
\end{equation*}
The Weyl law for the Laplacian (\cite[Chapter~9.5]{shubinInvitationPartialDifferential2020}) now yields
\begin{align*}
M_T (\gamma) & = \frac{\omega_d \pazocal{H}^d(\Omega)}{(2\pi)^d}\left[{\gamma^{-2}-1}\right]^{\frac{d}{2k}}  + o_{\gamma \to 0}\left(\left[{\gamma^{-2}-1}\right]^{\frac{d}{2k}}  \right) \\
&= \frac{\omega_d \pazocal{H}^d(\Omega)}{(2\pi)^d} \gamma^{-\frac{d}{k}} + o_{\gamma \to 0}\left(\gamma^{-\frac{d}{k}} \right).
\end{align*}
Upon application of Corollary \ref{cor: Metric entropy for polynomial ellipsoids second order ecf}, with the choices 
\begin{equation*}
\kappa_1 = \frac{\omega_d \pazocal{H}^d(\Omega)}{(2\pi)^d}, \ \
\kappa_2 = 0, \quad \text{and} \quad 
\beta_1=\beta_2 = \frac{d}{k},
\end{equation*}
we obtain the first desired result, namely 
\begin{equation*}
H \left(\varepsilon\right) 
= \frac{k \, \omega_d \,  \pazocal{H}^d(\Omega)}{d \, (2\pi)^d \, \ln(2)} \varepsilon^{-\frac{d}{k}} + o_{\varepsilon \to 0}\left(\varepsilon^{-\frac{d}{k}}\right)
\end{equation*}
and 
\begin{equation*}
\varepsilon_m
= \left[\frac{k \, \omega_d \,  \pazocal{H}^d(\Omega)}{d \,  (2\pi)^d \, \ln(2)}\right]^{\frac{k}{d}}{m^{-\frac{k}{d}}} + o_{m\to \infty}\left({m^{-\frac{k}{d}}}\right).
\end{equation*}
Under the assumptions of (ii) in the theorem statement, the two-term Weyl law for the Laplacian (see \cite{ivriiSecondTermSpectral1980}) allows us to conclude that 
\begin{align*}
M_T (\gamma)
&= \frac{\omega_d \pazocal{H}^d(\Omega)}{(2\pi)^d} \left[{\gamma^{-2}-1}\right]^{\frac{d}{2k}} \\
& \quad \ -\frac{\omega_{d-1} \pazocal{H}^{d-1}\left(\partial \Omega\right)}{4(2\pi)^{d-1}}
\left[{\gamma^{-2}-1}\right]^{\frac{d-1}{2k}}
+ o_{\gamma \to 0}\left(\gamma^{-\frac{d-1}{k}} \right)\\
&=
 \frac{\omega_d \pazocal{H}^d(\Omega)}{(2\pi)^d}\, \gamma^{-\frac{d}{k}} 
-\frac{\omega_{d-1} \pazocal{H}^{d-1}\left(\partial \Omega\right)}{4(2\pi)^{d-1}}\,
\gamma^{-\frac{d-1}{k}}
+ o_{\gamma \to 0}\left(\gamma^{-\frac{d-1}{k}} \right).
\end{align*}
It can now readily be verified that, under the choices
\begin{equation}\label{eq:thischoiceworkswell}
\kappa_1 = \frac{\omega_d \pazocal{H}^d(\Omega)}{(2\pi)^d}, \ \
\kappa_2 = -\frac{\omega_{d-1} \pazocal{H}^{d-1}\left(\partial \Omega\right)}{4(2\pi)^{d-1}}, \ \ 
\beta_1 = \frac{d}{k}, \quad \text{and} \quad 
\beta_2 = \frac{d-1}{k},
\end{equation}
we have $\beta_2<\beta_1$
and the assumption $d\geq 3$ implies $\beta_1<2\beta_2$.
We can hence apply Corollary \ref{cor: Metric entropy for polynomial ellipsoids second order ecf} with $\kappa_1$, $\kappa_2$, $\beta_1$, and $\beta_2$ according to (\ref{eq:thischoiceworkswell})
to obtain the second desired result 
\begin{equation*}
H \left(\varepsilon\right) 
= \frac{k \, \omega_d \,  \pazocal{H}^d(\Omega)}{d \, (2\pi)^d \, \ln(2)} \varepsilon^{-\frac{d}{k}}  
- \frac{k\, \omega_{d-1} \,  \pazocal{H}^{d-1}\left(\partial \Omega\right)}{4\,(d-1)\,(2\pi)^{d-1}  \, \ln(2)}
\varepsilon^{-\frac{d-1}{k}} 
+ o_{\varepsilon \to 0}\left(\varepsilon^{-\frac{d-1}{k}}\right)
\end{equation*}
and 
\begin{align*}
\varepsilon_m 
=& \left[\frac{k \, \omega_d \,  \pazocal{H}^d(\Omega)}{d \, (2\pi)^d \, \ln(2)} \right]^{\frac{k}{d}}{m^{-\frac{k}{d}}}\\
& \ - \frac{\pi \, k\, {\omega_{d-1} \pazocal{H}^{d-1}(\partial \Omega)}}{2 \, (d-1)\, \omega_d \, \pazocal{H}^d(\Omega) }\,\left(\frac{k\, \omega_d \pazocal{H}^d(\Omega)}{d\, (2\pi)^d \,  \ln(2)}\right)^{\frac{k+1}{d}} m^{-\frac{k+1}{d}}
+ o_{m\to \infty}\left({m^{-\frac{k+1}{d}}}\right),
\end{align*}
thereby concluding the proof.
\end{proof}

\bibliography{main}

\appendix

\section*{Appendix}\label{sec:app}

\begin{lemma}\label{lem:corpolydecay}
Let $\{\mu_n\}_{n \in \mathbb{N}^*}$ be a sequence of positive real numbers
such that 
\begin{equation*}
\mu_n 
= \frac{c_1}{n^{\alpha_1}} + \frac{c_2}{n^{\alpha_2}} + o_{n \to \infty} \left(\frac{1}{n^{\alpha_2}}\right),
\end{equation*}
with $c_1\in\mathbb{R}_+^*$, $c_2\in\mathbb{R}$, and $\alpha_1, \alpha_2\in \mathbb{R}_+^*$ satisfying 
\begin{equation*}
\text{either } \quad 
\alpha_1<\alpha_2 < \alpha_1+1/2,
\ \text{ or } \quad
\begin{cases}
\alpha_1 = \alpha_2, \text{ and } \\
c_2 = 0.
\end{cases} 
\end{equation*}
The metric entropy with respect to the $\|\cdot\|_2$-norm of the ellipsoid in $ \ell^2$ with semi-axes $\{\mu_n\}_{n \in \mathbb{N}^*}$ satisfies
\begin{equation*}
H \left(\varepsilon\right)
= \frac{\alpha_1{c_1}^{\frac{1}{\alpha_1}}}{\ln(2)} \, {\varepsilon}^{-\frac{1}{\alpha_1}} 
+ \frac{c_2\, {c_1}^{\frac{1-\alpha_2}{\alpha_1}}}{\ln(2)(\alpha_1-\alpha_2+1)}\, \varepsilon^{-\frac{\alpha_1-\alpha_2+1}{\alpha_1}} 
+ o_{\varepsilon \to 0}\left(\varepsilon^{-\frac{\alpha_1-\alpha_2+1}{\alpha_1}}\right).
\end{equation*}
\end{lemma}

\begin{proof}
For $\alpha_1 = \alpha_2$ and $c_2 = 0$, the result is
by \cite[Corollary~16]{secondpaper}; for  $\alpha_1<\alpha_2 < \alpha_1+1/2$,
the statement is a consequence of \cite[Theorem~17]{secondpaper}.
\end{proof}

\begin{lemma}[Inversion lemma, first order]\label{lem:inv1}
For $\kappa_1\in\mathbb{R}_+^*$ and $\beta_1 \in \mathbb{R}^*_+$, let
$\{\zeta_n\}_{n\in\mathbb{N}^*}$ be a sequence of real numbers satisfying 
\begin{equation*}
  n 
  = \kappa_1\zeta_n^{-\beta_1} + o_{n \to \infty} \left(\zeta_n^{-\beta_1}\right).
\end{equation*}
Then, we have 
\begin{equation*}
  \zeta_n
  = \kappa_1^{1/\beta_1} n^{-1/\beta_1}
  + o_{n \to \infty} \left(n^{-1/\beta_1}\right).
\end{equation*}
\end{lemma}

\begin{proof}
  The proof follows directly by rearranging terms and using the Taylor series expansion of $(\kappa_1+x)^{1/\beta_1}$ around $x=0$.
\end{proof}

\begin{lemma}[Inversion lemma, second order]\label{lem:inv}
For $\kappa_1\in\mathbb{R}_+^*$, $\kappa_2\in\mathbb{R}$,  and $\beta_1, \beta_2 \in \mathbb{R}^*_+$ such that $\beta_1 > \beta_2$, let
$\{\zeta_n\}_{n\in\mathbb{N}^*}$ be a sequence of positive real numbers satisfying 
\begin{equation}\label{eq:yelphe}
n 
= \kappa_1\zeta_n^{-\beta_1} + \kappa_2 \, \zeta_n^{-\beta_2} + o_{n \to \infty} \left(\zeta_n^{-\beta_2}\right).
\end{equation}
Then, we have 
\begin{equation*}
\zeta_n
= \kappa_1^{1/\beta_1} n^{-1/\beta_1}+ \frac{\kappa_1^{1/\beta_1-\beta_2/\beta_1}\kappa_2}{\beta_1} \, n^{\beta_2/\beta_1-1/\beta_1-1} + o_{n \to \infty} \left(n^{\beta_2/\beta_1-1/\beta_1-1}\right).
\end{equation*}
\end{lemma}

\begin{proof}
We first note that owing to $\beta_1, \beta_2 \in \mathbb{R}^*_+$, (\ref{eq:yelphe}) implies $\lim_{n\rightarrow \infty} \zeta_{n} = 0$. As $\beta_1 > \beta_2$, it hence follows that
\begin{equation*}
n 
= \kappa_1\zeta_n^{-\beta_1} 
+ o_{n \to \infty} \left(\zeta_n^{-\beta_1}\right),
\end{equation*}
which, upon application of Lemma~\ref{lem:inv1}, yields
\begin{equation}\label{eq:lkjhgfdd3}
\zeta_n
= \kappa_1^{1/\beta_1} n^{-1/\beta_1}\left(1+o_{n\to\infty}(1)\right).
\end{equation}
Therefore, upon rewriting
\begin{align*}
n 
&= \kappa_1\zeta_n^{-\beta_1} + \kappa_2 \, \zeta_n^{-\beta_2} + o_{n \to \infty} \left(\zeta_n^{-\beta_2}\right) \\
&= \kappa_1\zeta_n^{-\beta_1}\left[1+ \kappa_1^{-1}\kappa_2 \, \zeta_n^{\beta_1-\beta_2} + o_{n \to \infty} \left(\zeta_n^{\beta_1-\beta_2}\right)\right],
\end{align*}
we obtain 
\begin{align*}
\zeta_n
&\hspace{5pt} = \kappa_1^{1/\beta_1} n^{-1/\beta_1}\left[1+ {\kappa_1^{-1}\kappa_2} \, \zeta_n^{\beta_1-\beta_2} + o_{n \to \infty} \left(\zeta_n^{\beta_1-\beta_2}\right)\right]^{1/\beta_1}\\
&\hspace{5pt} = \kappa_1^{1/\beta_1} n^{-1/\beta_1}\left[1+ \frac{\kappa_1^{-1}\kappa_2}{\beta_1} \, \zeta_n^{\beta_1-\beta_2} + o_{n \to \infty} \left(\zeta_n^{\beta_1-\beta_2}\right)\right]\\
&\stackrel{(\ref{eq:lkjhgfdd3})}{=} \kappa_1^{1/\beta_1} n^{-1/\beta_1}\left[1+ \frac{\kappa_1^{-\beta_2/\beta_1}\kappa_2}{\beta_1} \, n^{\beta_2/\beta_1-1} + o_{n \to \infty} \left(n^{\beta_2/\beta_1-1}\right)\right]\\
&\hspace{5pt}= \kappa_1^{1/\beta_1} n^{-1/\beta_1}+ \frac{\kappa_1^{1/\beta_1-\beta_2/\beta_1}\kappa_2}{\beta_1} \, n^{\beta_2/\beta_1-1/\beta_1-1} + o_{n \to \infty} \left(n^{\beta_2/\beta_1-1/\beta_1-1}\right),
\end{align*}
which finalizes the proof.
\end{proof}

\begin{lemma}\label{lem:jkdqvbftgzdezrfgezg}
  For $\alpha_1,c_1\in\mathbb{R}_+^*$, let $\{\lambda_n\}_{n\in\mathbb{N}^*}$ be a sequence of positive real numbers satisfying $\lambda_n \stackrel{n\to\infty}{\sim} c_1 n^{-\alpha_1}$.
  Then, we have
  \begin{equation*}
    \sup_{N \in \mathbb{N}^*} \left\{2^{-m/N} \left[\prod_{n=1}^N \lambda_n \right]^{1/N} \right\}
    = c_1  \left(\frac{\alpha_1}{\ln(2)}\right)^{\alpha_1}m^{-\alpha_1} + o_{m\to\infty}\left({m^{-\alpha_1}}\right).
  \end{equation*}
\end{lemma}

\begin{proof}
    Using $\lambda_n \stackrel{n\to\infty}{\sim} c_1 n^{-\alpha_1}$, we get
  \begin{align}
      \ln\left(\left[\prod_{n=1}^N \lambda_n \right]^{1/N} \right) 
      &= {\frac{1}{N}\sum_{n=1}^N\ln\left(c_1 n^{-\alpha_1}\left(1+o_{n\to\infty}(1)\right)\right)} \label{eq:hdfjkshjkbdjkbvjkbjsf7893}\\
      &= {\frac{1}{N}\sum_{n=1}^N\left[\ln\left(c_1 n^{-\alpha_1}\right)+o_{n\to\infty}(1)\right]} \nonumber\\ 
      &= {\frac{1}{N}\sum_{n=1}^N\left[\ln\left(c_1 n^{-\alpha_1}\right)\right]+o_{N\to\infty}(1)} \nonumber\\
      &= \ln\left(\left[\prod_{n=1}^N (c_1 n^{-\alpha_1})\right]^{1/N} \right)+o_{N\to\infty}(1),\label{eq:hdfjkshjkbdjkbvjkbjsf7894}
  \end{align}
    where, in the penultimate step, we applied \cite[Lemma~42]{secondpaper} with $a=0$.
    Taking exponents in (\ref{eq:hdfjkshjkbdjkbvjkbjsf7893})-(\ref{eq:hdfjkshjkbdjkbvjkbjsf7894}) and applying Stirling's formula, yields
  \begin{equation}\label{eq:lkljjllkzlkdpzddekdekcvkjfndkjvnlkfnsd}
      \left[\prod_{n=1}^N \lambda_n \right]^{1/N} 
      \stackrel{N\to\infty}{\sim} c_1 \left(\frac{N}{e}\right)^{-\alpha_1}.
  \end{equation}
  For fixed $m\in\mathbb{N}^*$, let the integer $N$ that maximizes $2^{-m/N} [\prod_{n=1}^N \lambda_n ]^{1/N}$ be denoted as $N^*_m\in\mathbb{N}^*$. 
  We remark that $N^*_m \xrightarrow{m\to\infty}\infty$ as otherwise
  $2^{-m/N^*_m} [\prod_{n=1}^{N^*_m} \lambda_n]^{1/N^*_m}$ would decay exponentially fast in $m$, 
  while, e.g., the choice $N=m$ delivers polynomial decay.
  It now follows from (\ref{eq:lkljjllkzlkdpzddekdekcvkjfndkjvnlkfnsd}) that
  \begin{equation}\label{eq:jkvsfkdgvbjkbsdkjkjjkrnfvskjvf1}
    2^{-m/N^*_m} \left[\prod_{n=1}^{N^*_m} \lambda_n \right]^{1/N^*_m}
    \stackrel{m\to\infty}{\sim} c_1 2^{-m/N^*_m} \left(\frac{N^*_m}{e}\right)^{-\alpha_1}.
  \end{equation}
  Analyzing the function 
  \begin{equation*}
    t \in \mathbb{R}_+^*
    \longmapsto \frac{m}{t}+\alpha_1\log(t),
  \end{equation*}
  shows that $N^*_m \stackrel{m\to\infty}{\sim} \ln(2)\, m / \alpha_1$, which, in turn, yields
  \begin{equation}\label{eq:jkvsfkdgvbjkbsdkjkjjkrnfvskjvf2}
    c_1 2^{-m/N^*_m} \left(\frac{N^*_m}{e}\right)^{-\alpha_1}
    \stackrel{m\to\infty}{\sim} c_1  \left(\frac{\alpha_1}{\ln(2)}\right)^{\alpha_1}m^{-\alpha_1}.
  \end{equation}
  Combining (\ref{eq:jkvsfkdgvbjkbsdkjkjjkrnfvskjvf1}) and (\ref{eq:jkvsfkdgvbjkbsdkjkjjkrnfvskjvf2}), finalizes the proof.
\end{proof}

\begin{lemma}\label{lem:lbentropy}
For $d\in \mathbb{N}^*$, let
 $\pazocal{E}^d$ be the $\|\cdot\|_2$-ellipsoid in $\mathbb{C}^d$ with positive semi-axes $\mu_1, \dots, \mu_d$.
 The metric entropy of $\pazocal{E}^d$ with respect to the $\|\cdot\|_2$-norm satisfies
\begin{equation*}
H \left(\varepsilon ; \pazocal{E}^d, \|\cdot\|_2 \right)
\geq 2  d \left[\log\left(\varepsilon^{-1}\right)  + \frac{1}{d}\sum_{n=1}^d \log \left(\mu_n\right) \right], \quad \mbox{for all}\,\, \varepsilon>0. 
\end{equation*}
\end{lemma}

\begin{proof}
The proof follows directly by application of \cite[Theorem~4]{firstpaper} with $p=q=2$, $\mathbb{K}=\mathbb{C}$, and where we set $\underline{\kappa}=1$, which is possible by the remark at the end of the proof of \cite[Theorem~4]{firstpaper}.
\end{proof}

\begin{lemma}\label{lem:densityarg}
Let $\pazocal{B}_2$ be the unit ball in $\mathbb{C}^d$ with respect to the $\|\cdot\|_2$-norm.
Then, we have 
\begin{equation*}
N\left(\varepsilon; \pazocal{B}_2, \|\cdot\|_2 \right)^{1/(2d)} \, \varepsilon
\leq \kappa(d),
\quad \text{for all } d \in \mathbb{N}^* \text{ and } \varepsilon \in (0,1),
\end{equation*}
with the sequence $\kappa(d), d \in \mathbb{N}^*$, satisfying 
$\lim_{d\to\infty}\kappa(d) = 1$.
\end{lemma}

\begin{proof}
We identify $\mathbb{C}^d$ with $\mathbb{R}^{2d}$.
For $d=1, \dots, 4$, we simply set 
$$
\kappa(d) = \sup_{\varepsilon \in (0,1)} N(\varepsilon; \pazocal{B}_2, \|\cdot\|_2 )^{1/(2d)} \, \varepsilon.
$$
This supremum is finite as a consequence of \cite[Lemma~3]{firstpaper} with $\mathbb{K}=\mathbb{C}$, $\|\cdot\|=\|\cdot\|'=\|\cdot\|_2$, and $\mathcal{B}=\mathcal{B}'=\mathcal{B}_2$, which states
\begin{equation*}
    N(\varepsilon; \pazocal{B}_2, \|\cdot\|_2 )^{1/(2d)} \, \varepsilon
    \leq 2 \left(1+\varepsilon/2\right),
    \quad \text{for all } \varepsilon\in (0,1).
\end{equation*}
For $d\geq 5$, the statement follows by \cite[Theorem~3]{rogers1963covering} with $\varepsilon = 1/R$ and upon setting $\kappa(d) = c(d^{5/2})^{1/(2d)}$, for some constant $c>0$.
\end{proof}

\end{document}